\documentclass[reqno, 10.5pt]{amsart}
\usepackage{amsmath}
\usepackage{amscd}
\usepackage{amssymb}

\newtheorem{theorem}{Theorem}[section]
\newtheorem{theorem/definition}{Theorem/Definition}[section]
\newtheorem{proposition}{Proposition}[section]
\newtheorem{lemma}{Lemma}[section]

\theoremstyle{remark}
\newtheorem{remark}{Remark}[section]
\theoremstyle{definition}

\begin{document}
\title[Curvature Estimates for four-dimensional gradient Steady Solitons]
{Curvature Estimates for four-dimensional gradient Steady Ricci Solitons}
\author{HUAI-DONG CAO AND XIN CUI}
\address{Department of Mathematics, University of Macau, Macao \& Lehigh University,
Bethlehem, PA 18015}
\email{huc2@lehigh.edu}

\address{Department of Mathematics\\ Lehigh University\\
Bethlehem, PA 18015}
\email{xic210@lehigh.edu}


\begin{abstract}
In this paper, we derive certain curvature estimates for 4-dimensional gradient steady Ricci solitons either with positive Ricci curvature or with scalar curvature decay $\lim_{x\to \infty} R(x)=0$.

\end{abstract}

\maketitle
\date{}

\section{The results}

A complete Riemannian metric $g_{ij}$ on a smooth manifold $M^n$
is called a {\it gradient steady Ricci soliton} if there exists
a smooth function $f$ on $M^n$ such that the Ricci tensor $R_{ij}$
of the metric $g_{ij}$ satisfies the equation
$$R_{ij}+\nabla_i\nabla_jf=0. \eqno(1.1)$$
The function $f$ is called a {\it potential function} of the gradient steady soliton.
Clearly, when $f$ is a constant the gradient steady Ricci soliton $(M^n, g_{ij}, f)$ is simply a
Ricci-flat manifold. Gradient steady solitons play an important role in
Hamilton's Ricci flow, as they correspond to translating
solutions, and often arise as Type II singularity models. Thus one is interested
in possibly classifying them or understanding their geometry.

It turns out that compact steady solitons must be Ricci-flat. In dimension $n=2$, Hamilton \cite{Ha88} discovered the
first example of a complete noncompact gradient steady
soliton on $\mathbb R^2$, called the {\it cigar soliton}, where the metric is given by
$$ ds^2=\frac{dx^2 +dy^2}{1+x^2+y^2}.$$
The cigar soliton has potential function $f=-\log (1+x^2+y^2)$, positive curvature $R=4e^{f}$, and is asymptotic to cylinder at infinity.  Furthermore, Hamilton \cite{Ha88} showed that the only
complete steady soliton on a two-dimensional manifold with
bounded (scalar) curvature $R$ which assumes its maximum
at an origin is, up to scaling,  the cigar soliton. For $n\geq 3$, Bryant \cite{Bryant} proved that
there exists, up to scaling, a unique complete rotationally symmetric gradient Ricci
soliton on $\Bbb R^n$, see, e.g.,  Chow et al. \cite{Chow et al 1}
for details. The Bryant soliton has positive sectional curvature, linear curvature decay
and volume growth of geodesic balls $B(0,r)$ on the order of $r^{(n+1)/2}$. In the K\"ahler case,
Cao \cite{Cao94} constructed a complete $U(m)$-invariant gradient steady K\"ahler-Ricci soliton on $\mathbb{C}^m$, for $m\geq 2$, with positive sectional curvature. It has volume growth on the order of $r^{m}$ and also linear curvature decay. Note that in each of these three examples, maximum of the scalar curvature is attained at the origin. One can find additional examples of steady solitons, e.g., in \cite {Iv, FIK, DW1, DW2, BDGW} etc; see also \cite{Cao08b} and the references therein.

In dimension $n=3$, Perelman \cite{P2} claimed that the Bryant soliton is the
only complete noncompact, $\kappa$-noncollapsed, gradient steady
soliton with positive sectional curvature.  Recently, Brendle has affirmed this conjecture of Perelman
(see \cite{Brendle 1}; and also \cite{Brendle 2} for an extension to the higher dimensional case).
On the other hand, for $n\ge 4$,  Cao-Chen \cite{CaoChen} and Catino-Mantegazza \cite{CM} proved independently, and using different methods, that any $n$-dimensional  complete noncompact locally conformally flat gradient steady Ricci soliton $(M^n, g_{ij}, f)$ is either flat or isometric to the Bryant soliton (the method of
Cao-Chen \cite{CaoChen} also applies to the case of dimension $n=3$). In addition, Bach-flat gradient steady solitons (with positive Ricci curvature) for all $n\ge 3$ \cite{Cao et al} and half-conformally
flat ones for $n=4$ \cite{CW} have been classified respectively.

Inspired by the very recent work of Munteanu-Wang \cite{MW}, in this paper we study curvature estimates of 4-dimensional complete noncompact gradient steady solitons.
In \cite{MW}, Munteanu and Wang made an important observation that the curvature tensor of a four-dimensional gradient Ricci soliton $(M^4, g_{ij}, f)$ can be estimated in terms of the potential function $f$, the Ricci tensor and its first derivates. In addition, the (optimal) asymptotic quadratic growth property of the potential function $f$ proved in \cite{CaoZhou}, as well as a key scalar curvature lower bound $R\ge c/f$ shown in \cite{CLY} are crucial in their work. While gradient steady Ricci solitons in general don't share these two special features (cf. \cite{Wu, MS, WW} and \cite{CLY, FG}), some of the arguments in \cite{MW} can still be adapted to prove certain curvature estimates for two classes of gradient steady solitons.
Our main results are:

\begin{theorem} Let $(M^4, g_{ij}, f)$ be a complete noncompact $4$-dimensional
gradient steady Ricci soliton with positive Ricci curvature $Ric>0$ such that the scalar curvature $R$ attains its maximum
at some point $x_0\in M^4$. Then, $(M^4, g_{ij})$ has bounded Riemann curvature tensor, i.e.
$$ \sup_{x\in M} |Rm|  \le C $$
for some constant $C>0$. Suppose in addition $R$ has at most linear decay, then $$ \sup_{x\in M}\frac{|Rm|}{R}  \le C.$$

\end{theorem}

\begin{theorem} Let $(M^4, g_{ij}, f)$, which is not Ricci-flat,  be a  complete noncompact $4$-dimensional
gradient steady Ricci soliton. If $\lim_{x\to \infty} R(x)=0$, then, for each $0<a<1$, there exists
a constant $C>0$ such that
$$  |Ric|^2\le C {R^a} \qquad  \mbox{and} \qquad \sup_{x\in M} |Rm| \le C.$$
Suppose in addition $R$ has at most polynomial decay. Then, for each $0<a<1$, there exists a constant $C>0$ such that
$$  |Rm|^2\le C R^{a}.$$

\end{theorem}

\begin{remark}
We have learned that the first part of Theorem 1.2 was also known to O. Munteanu and J. Wang.
\end{remark}

\medskip \noindent {\bf Acknowledgments.} We are grateful to Ovidiu Munteanu and Jiaping Wang for sending us their paper \cite{MW}, and its early version in July 2014,  which motivated us to consider curvature
estimates for 4D steady solitons.  The first author also would like to thank Ovidiu Munteanu for very helpful discussions; part of the work was carried out when he was visiting University of Macau, partially support by FDCT/016/2013/A1 and RDG010, during summer 2014.

\section{Preliminaries}

In this section, we recall some basic facts and collect several known results about gradient steady
solitons. Throughout the rest of the paper, we denote by $$Rm=\{R_{ijkl}\}, \quad Ric=\{R_{ik}\},\quad R $$ the Riemann curvature tensor, the Ricci tensor, and the scalar curvature of the metric $g_{ij}$ respectively.

\begin{lemma} {\bf (Hamilton \cite{Ha95F})} Let $(M^n, g_{ij}, f)$
be a complete gradient stead soliton satisfying Eq. (1.1).
Then
$$ R=-\Delta f, \eqno (2.1)$$
$$\nabla_iR=2R_{ij}\nabla_jf, \eqno(2.2)$$
$$R+|\nabla f|^2=C_0 \eqno(2.3)$$ for some constant $C_0$.
\end{lemma}

Next we need the following useful result by B.-L. Chen \cite{BChen}.

\begin{proposition}  Let $g_{ij}(t)$ be a complete ancient
solution to the Ricci flow on a noncompact manifold $M^n$. Then the scalar curvature
$R$ of $g_{ij}(t)$ is nonnegative for all $t$.
\end{proposition}

As an immediate corollary, we have

\begin{lemma}
Let $(M^n, g_{ij}, f)$ be a complete gradient steady soliton. Then it has nonnegative
scalar curvature $R\ge 0$.
\end{lemma}

\begin{remark}
In fact, by Proposition 3.2 in \cite {PeW},  either $R>0$ or $(M^n, g_{ij})$ is Ricci flat.
\end{remark}

It follows from Lemma 2.2 that the constant $C_0$ in (2.3) is positive whenever $f$ is a non-constant
function (i.e., the steady soliton is non-trivial). By a suitable scaling of the metric $g_{ij}$, we can normalize $C_0=1$ so that
$$R+|\nabla f|^2=1. \eqno(2.4)$$
In the rest of the paper, we shall always assume this normalization (2.4).

Combining  (2.1) and (2.4), we obtain $-\Delta f +|\nabla f|^2=1.$ Thus, setting $F=-f$, we have
$$\Delta_f F =1, \eqno(2.5)$$
where $$\Delta_f =:\Delta -\nabla f\cdot \nabla. \eqno (2.6)$$

For gradient steady solitons with positive Ricci curvature $Ric>0$, we have

\begin{proposition} Let $(M^n, g_{ij}, f)$ be a complete noncompact
gradient steady soliton with positive Ricci curvature $Ric>0$ such that the scalar curvature $R$ attains its maximum $R_{max}=1$ at some origin
$x_0\in M^n$. Then, there exist some constants $0<c_1\leq 1$ and $c_2>0$  such that $F=-f$ satisfies the estimates
$$c_1r(x)-c_2 \leq F(x) \leq r(x) + |F(x_0)|, \eqno(2.7)$$
where $r(x)=d(x_0, x)$ is the distance function from  $x_0$.
\end{proposition}

\begin{remark} In (2.7), only the lower bound on $F$ requires the assumptions on $Ric$ and $R$.
\end{remark}

Note that, under the assumption in Proposition 2.2, $F(x)$ is proportional to the distance function $r(x)=d(x_0, x)$ from above and below. Throughout the paper, we denote
\begin{eqnarray*}
D(t)&=& \{x\in M : F(x) \leq t \},\\ 
B(t)=B(x_0, t)&=&\{x\in M : d(x_0,x) \leq t \}.
\end{eqnarray*}

We also collect several differential equalities on $R, Ric$ and $Rm$ which are special cases of evolution equations of the curvature tensors under the Ricci flow:

\begin{lemma} Let $(M^n, g_{ij}, f)$
be a complete gradient steady soliton satisfying Eq. (1.1). Then, we have
\begin{eqnarray*}
\Delta_{f} R &=&-2|Ric|^2,\\
\Delta_{f} Ric &=& -2R_{ijkl}R_{jl},\\
\Delta_{f} {Rm} &= & Rm\ast Rm,
\end{eqnarray*}
where the RHS of the last equation denotes (a finite number of) terms involving quadratics in $Rm$.
\end{lemma}

Based on Lemma 2.3,  one can easily derive the following inequalities similar to \cite{MW}:

\begin{lemma} Let $(M^n, g_{ij}, f)$ be a complete gradient steady soliton satisfying Eq. (1.1). Then
\begin{eqnarray*}
\Delta_{f} |Ric|^2 & \ge & 2|\nabla Ric|^2-4|Rm| |Ric|^2, \\
\Delta_{f} |Rm| &\ge & -c |Rm|^2, \\
\Delta_{f}|Rm|^2 & \ge & 2|\nabla Rm|^2 - C|Rm|^3.
\end{eqnarray*}
Here $c>0$ is some universal constant depending only on the dimension $n$.
\end{lemma}

\begin{remark} To derive the second differential inequality, one needs to use the Kato  inequality
$|\nabla |Rm||\le |\nabla Rm|$ as shown in \cite{MW}.
\end{remark}

\section{4D gradient steady solitons with positive Ricci curvature}

First of all, we need the following key fact, valid for 4-dimensional gradient steady Ricci solitons in general,  due to Munteanu and Wang \cite{MW}.

\begin{lemma}{\bf (Munteanu-Wang \cite{MW})} Let $(M^4, g_{ij}, f)$ be a complete noncompact gradient steady soliton satisfying (1.1).
Then, there exists some universal constant $c>0$ such that
$$ |Rm|\le c \large( \frac {|\nabla Ric|} {|\nabla f|} + \frac{|Ric|^2}{|\nabla f|^2} +|Ric| \large). $$
\end{lemma}

\begin{proof} This follows from the same arguments as in the proof of Proposition 1.1 of \cite{MW}, but using Lemma 2.4 instead and without replacing $|\nabla f|^2$ by $f$ in their argument.

\end{proof}

\begin{proposition} Let $(M^4, g_{ij}, f)$ be a complete noncompact gradient steady soliton with positive Ricci curvature and $R$ attains maximum.
Then, there exists some constant $C>0$, depending on the constant $c_1$ in (2.7), such that outside a compact set,
$$ |Rm|\le C \large( |\nabla Ric| + |Ric|^2 +|Ric| \large). $$
\end{proposition}

\begin{proof} This easily follows from Lemma 3.1 and the following fact shown by Cao-Chen \cite {CaoChen},
$$|\nabla f|^2\ge c_1>0. \eqno(3.1)$$
\end{proof}

\begin{remark}
Note that, combining (3.1) with (2.4) and Lemma 2.2, we have
$$0< c_1 \le |\nabla F|^2=|\nabla f|^2\le 1.\eqno(3.2)$$
\end{remark}

\medskip
Now we are ready to prove the first main result of our paper.

\begin{theorem}  Let $(M^4, g_{ij}, f)$ be a complete noncompact gradient steady soliton with positive Ricci curvature $Ric>0$ such that $R$ attains its maximum
at some point $x_0\in M^4$.
Then, there exists some constant $C>0$, depending on $c_1$ in (2.7), such that
$$ \sup_{x\in M} |Rm|  \le C. $$

\end{theorem}

\begin{proof} First of all, from (2.4), we have $R\le 1$. Hence, since $Ric>0$, it follows that
$$ 0<|Ric|\le R\le 1. \eqno(3.3)$$
Thus, by Proposition 3.1 and (3.3), we see that
$$|\nabla Ric|^2\geq \frac{1}{2C^2}|Rm|^2 -(|Ric|^2+|Ric|)^2\ge  \frac{1}{2C^2}|Rm|^2-4. \eqno(3.4)$$
Using the first two inequalities in Lemma 2.4, we obtain
$$\Delta_f(|Rm|+\lambda |Ric|^2) \geq -C |Rm|^2 + 2\lambda (|\nabla Ric|^2-2|Rm| |Ric|^2). \eqno(3.5)$$
By (3.3), (3.4), and picking constant $\lambda>0$ sufficiently large
(depending on the constant $C$ in Proposition 3.1, hence on $c_1$), it follows that
$$
\Delta_f(|Rm|+\lambda |Ric|^2) \ge 2|Rm|^2-4\lambda |Rm|-C'
\ge (|Rm|+\lambda |Ric|^2)^2 -C. \eqno(3.6)
$$

Next, let $\varphi (t)$ be a smooth function on $\mathbb R^{+}$ so that $0\le \varphi (t)\le 1$, $\varphi (t)=1$ for
$0\le t\le R_0$, $\varphi (t)=0$ for
$t\ge 2R_0$, and
$$ t^2 \left(|\varphi'(t)|^2+|\varphi''(t)|\right)\le c \eqno(3.7)$$
for some universal constant $c$ and $R_0>0$ arbitrary large. We now take $\varphi=\varphi (F(x))$ as a cut-off function with support in $D(2R_0)$. Note that 
$$ |\nabla \varphi|=|\varphi'||\nabla F|\le  \frac {c} {R_0} \quad  \mbox{and} \quad |\Delta_f \varphi| \le |\varphi' \Delta_f F| +|\varphi''||\nabla F|^2\le \frac c R_0\eqno(3.8)$$
on $D(2R_0)\setminus D(R_0)$ for some universal constant $c$.

Setting $u = |Rm|+\lambda |Ric|^2$ and $G=\varphi^2 u$, then direct computations, (3.6) and (3.8)
yield
\begin{eqnarray*} \varphi^2 \Delta_f G &=&\varphi^4 \Delta_f u +\varphi^2u \Delta_f (\varphi^2) +2\varphi^2\nabla u\cdot\nabla \varphi^2 \\
&\ge &\varphi^4 \left( u^2 -C\right) + \varphi^2 u \left(2\varphi\Delta_f \varphi +2|\nabla\varphi|^2\right) +2\nabla G\cdot\nabla \varphi^2 -8|\nabla \varphi|^2 G \\
&\ge & G^2+2\nabla G\cdot\nabla \varphi^2  -C G -C.
\end{eqnarray*}
Now it follows from the maximum principle that  $G\le C$ on $D(2R_0)$ by some constant $C>0$  depending on $c_1$ but independent of $R_0$. Hence $u=|Rm|+\lambda |Ric|^2\le C$ on $D(R_0)$. Since $R_0>0$ is arbitrary large, we see that
$$ \sup_{x\in M}|Rm|\le \sup_{x\in M} \left(|Rm|+\lambda |Ric|^2\right)\le C.$$
This completes the proof of Theorem 3.1.
\end{proof}

\begin{proposition} Let $(M^4, g_{ij}, f)$ be a complete noncompact gradient steady soliton with positive Ricci curvature $Ric>0$ and $R$ attains its maximum at $x_0\in M^4$.
Then, function $u=\frac{|Rm| +\lambda |Ric|^2}{R}$, with $\lambda>0$ sufficiently large, satisfies the differential inequality
$$\Delta_f u \ge u^2R -CR -2\nabla u\cdot\nabla (\log R)$$
for some constant  $C>0$.

\end{proposition}
\begin{proof}
First of all, similar to deriving (3.4)-(3.6) in the proof of Theorem 3.1, by choosing $\lambda$ sufficiently large we have
\begin{eqnarray*}
\Delta_f(|Rm|+\lambda |Ric|^2) &\ge& (|Rm|+\lambda |Ric|^2)^2 - 4\lambda^2 |Ric|^4 -\lambda (|Ric|^4 + |Ric|^2)\\
&\ge& (|Rm|+\lambda |Ric|^2)^2 - C|Ric|^2
\end{eqnarray*}
for some constant $C>0$. Here we have also used the fact (3.3).

Thus, by a direct computation,
\begin{eqnarray*}
\Delta_f u &= &  R^{-1} \Delta_f (|Rm|+\lambda |Ric|^2) + (uR) \Delta_f (R^{-1})
+2\nabla (u R)\cdot \nabla (R^{-1})\\
& \ge & \frac {(|Rm|+\lambda |Ric|^2)^2 - C|Ric|^2} {R} + (uR) \left[2\frac{|Ric|^2}{R^2} +2\frac {|\nabla R|^2}{R^3}\right] \\
& & -\frac {2}{R^2} \left(u |\nabla R|^2+R\nabla u\cdot \nabla R\right)\\
& \ge & Ru^2 -CR -2\nabla u \cdot \nabla \log R.
\end{eqnarray*}

\end{proof}

\begin{theorem}
Let $(M^4,g_{ij},f)$ be a complete noncompact gradient steady Ricci soliton with $Ric>0$ such that  $R$ attains its maximum. Suppose $R$ has at most linear decay, $R(x)\geq c/{r(x)}$, outside a compact set. Then $$\sup_{x\in M}\frac{|Rm|}{R}  \le C.$$
\end{theorem}

\begin{proof}
Fix $\lambda$ sufficient large so that Proposition 3.2 holds and set $u=\frac{|Rm|+\lambda|Ric|^2}{R}$.
Next, let $\varphi (t)$ be a smooth function on $\mathbb R^{+}$ so that $\varphi (t)=\frac{d-t}{d}$ for
$0\le t\le d$, $\varphi (t)=0$ for
$t\ge d$. Let $\varphi=\varphi(F)$ and $G=\varphi^2 u$. Then,
\begin{align*}
|\nabla \varphi| &=|\varphi' \nabla F|\le \frac{1}{d},\\
\Delta_f \varphi&=\varphi' \Delta_f F=-\frac{1}{d}.
\end{align*}

Then, outside $D(1)$,
we have
\begin{eqnarray*}
\varphi^2\triangle_f (G) &=&
\varphi^4 (\triangle_f u)+\varphi^2 u (\triangle_f\varphi^2) +2\varphi^2 \nabla\varphi^2\cdot\nabla u\\
&\geq & \varphi^4 \left( R u^2-cR-2 \nabla u\cdot\nabla\log R\right)\\
& &+2 \left(\varphi\triangle_f \varphi+ |\nabla\varphi|^2 \right)G+2\varphi^2 \nabla \varphi^2\cdot\nabla u\\
&\geq & R G^2-cR+ 4\varphi\left( \nabla\varphi\cdot\nabla\log R\right) G\\
& &-\frac{8}{d} G+\left(2\nabla\varphi^2-2\varphi^2\nabla\log R\right)\cdot \nabla G.\\
\end{eqnarray*}

Now by (2.2), (3.2) and $Ric>0$, we have $|\nabla \log R|=2|\frac{Ric(\nabla f)}{R}|\leq 2$. Also, when $R$ has at most linear decay outside some $D(t_0)$ and for $d>t_0$, we have $R\geq \frac{a}{d}$ in $D(d)\backslash D(t_0)$ for some constant $a>0$ independent $d$. Therefore there exists $c$ independent of $d$ so that,
\begin{eqnarray*}
\varphi^2\triangle_f (G) &\geq &
RG^2-cR -\frac c d G +\left(2\nabla\varphi^2-2\varphi^2\nabla\log R\right)\cdot \nabla G\\
&\ge & \frac 1 2 RG^2-cR+\left(2\nabla\varphi^2-2\varphi^2\nabla\log R\right)\cdot \nabla G.
\end{eqnarray*}
 Therefore, it follows from stardard maximum principle argument that  $u\le C$ on $M^4$, hence
$|Rm|\le CR$ on $M^4$.

\end{proof}

\section{The proof of Theorem 1.2}

In this section, we prove our second main result, Theorem 1.2 in the introduction.
Throughout the section we assume $(M^4, g_{ij}, f)$ is a complete noncompact, non Ricci-flat,
$4$-dimensional gradient steady Ricci soliton such that
$$\lim_{x\to \infty} R(x)=0. \eqno (4.1)$$
Note that, by Lemma 2.2 and Remark 2.1, $(M^4, g_{ij}, f)$ necessary has strictly positive scalar curvature $R>0$.

First of all, we need the following useful Laplacian comparison type result for gradient
Ricci solitons.

\begin{lemma} Let $(M^n, g_{ij}, f)$ be any gradient steady Ricci soliton satisfying (1.1)
and let $r(x)=d(x_0, x)$ denote the distance function on $M^n$
from a fixed base point $x_0$. Suppose that
$$ Ric \le (n-1)K$$ on the geodesic ball $B(x_0, r_0)$ for some
constants $r_0>0$ and $K>0$. Then, for any $x\in M^n\setminus B(x_0, r_0)$, we have
$$\Delta_f r(x) \le (n-1)\left(\frac 23 Kr_0 + r_0^{-1}\right).$$
\end{lemma}

\begin{remark} Lemma 4.1 is a special case of a more general result valid for solutions to the Ricci flow
due to Perelman \cite{P1}, see, e.g., Lemma 3.4.1 in \cite{CaoZhu}.  Also see \cite{FLZ} and \cite{Wei} for a different version.
\end{remark}

\medskip

\begin{theorem} Let $(M^4, g_{ij}, f)$, which is not Ricci-flat,  be a complete noncompact gradient steady Ricci soliton. If $\lim_{x\to \infty} R(x)=0$, then, for each $0<a<1$, there exists
a constant $C>0$ such that
$$ \sup_{x\in M} |Ric|^2\le C {R^a} \qquad  \mbox{and} \qquad \sup_{x\in M} |Rm| \le C.$$
\end{theorem}

\begin{proof} The proof is similar to that of Munteanu-Wang
\cite{MW} except we need to use  the distance function to cut-off rather than the
potential function since the potential function may not be proper.

Since $\lim_{x\to \infty} R(x)=0$, it follows from (2.4) that
$$|\nabla f|\ge c_1>0$$ for some $0<c_1<1$
outside a compact set. By Lemma 2.4 and  Lemma 3.1,  we have
\begin{eqnarray*}
\Delta_f|Ric|^2&\geq&2|\nabla Ric|^2-C|Rm||Ric|^2\\
&\geq&2|\nabla Ric|^2-C\left(|\nabla Ric| +|Ric|^2+|Ric|\right)|Ric|^2.
\end{eqnarray*}

Also, since  $R>0$ on $M^4$, by using the first identity in Lemma 2.3 we have

\begin{eqnarray*}
\Delta_f \left(\frac{1}{R^a}\right) &=&2a\frac{|Ric|^2} {R^{a+1}}+a(a+1)\frac{|\nabla R|^2}{R^{a+2}}.
\end{eqnarray*}

Hence,
\begin{eqnarray*}
\Delta_f \left(\frac{|Ric|^2}{R^{a}}\right)&=& \frac{\Delta_{f}|Ric|^2}{R^{a}}+|Ric|^2\Delta_{f}\left(\frac{1}{R^{a}}\right)+2\nabla|Ric|^2\cdot\nabla\left(\frac{1}{R^{a}}\right)\\
&\geq& \frac{2|\nabla Ric|^2}{R^{a}}-C\frac{\left(|\nabla Ric|+|Ric|^2+|Ric|\right)|Ric|^2}{R^{a}}\\
& &+|Ric|^2 \left[2a\frac{|Ric|^2} {R^{a+1}}+a(a+1)\frac{|\nabla R|^2}{R^{a+2}}\right]-4a\frac{|Ric|\left|\nabla |Ric|\right||\nabla R|}{R^{a+1}}
\end{eqnarray*}

Apply Cauchy's inequality to the last term
\begin{eqnarray*}
-4a\frac{|Ric|\left|\nabla |Ric|\right||\nabla R|}{R^{a+1}} &\ge& -4a\frac{|Ric|\left|\nabla Ric\right||\nabla R|}{R^{a+1}}\\
&\ge & -a(a+1) \frac{|Ric|^2|\nabla R|^2}{R^{a+2} } - \frac{4a}{a+1} \frac{|\nabla Ric|^2}{R^{a}}.
\end{eqnarray*}

Thus, we have
\begin{eqnarray*}
\Delta_f \left(|Ric|^2{R^{-a}}\right)
&\geq&\frac{2(1-a)} {1+a} \frac{|\nabla Ric|^2}{R^{a}}- C\frac{|\nabla Ric||Ric|^2}{R^a} \\
& & - C \frac{|Ric|^4+|Ric|^3}{R^a } +2a \frac{|Ric|^4}{R^{a+1}} \\
&\geq& (2a-\frac{CR}{1-a})\frac{|Ric|^4}{R^{a+1}}-C\frac{|Ric|^3}{R^a}.
\end{eqnarray*}

Therefore, for $u=\frac{|Ric|^2}{R^a}$, we have derived the differential inequality
$$\Delta_f u\ge  (2a-\frac{CR}{1-a}) u^2 R^{a-1} -C u^{3/2} R^{a/2}. \eqno (4.2)$$

By assumption (4.1), for any $0<a<1$, we can choose a fixed $d_0>0$ depending on $a$ and sufficiently large  so that
$$ (2a-\frac{CR}{1-a}) \ge a \eqno(4.3)$$
outside the geodesic ball $B(x_0, d_0)$.

Next, for any $D_0>2d_0$, we choose a function $\varphi (t)$ as follows:
$0\le \varphi (t)\le 1$ is a smooth  function on $\mathbb {R}$ such that
$$ {\varphi (t) =\left\{
       \begin{array}{ll}
  1, \ \ \quad  2d_0\le t\le D_0,\\[4mm]
    0, \ \ \quad  t\le d_{0} \ \mbox{or} \ t\ge 2D_0.
       \end{array}
    \right.}$$
Also,
$$ t^2 |\varphi''(t)|\le c \qquad \mbox{and} \qquad 0\ge \varphi'(t) \ge -\frac {c} {D_0} , \ \mbox {if} \ 2d_0\le t \le 2D_0. \eqno(4.4)$$

Now we  use $\varphi=\varphi (r(x))$ as a cut-off function whose support is in $B(x_0, 2D_0)\setminus B(x_0, d_0)$. Note that, by (4.4) and Lemma 4.1,  we get
$$ |\nabla \varphi|^2=|\varphi'|^{2}\le  \frac {c} {D_0^2} \quad  \mbox{and} \quad \Delta_f \varphi =\varphi' \Delta_f r(x) +\varphi''\ge -\frac{C}{D_0} . \eqno(4.5)$$
on $B(x_0, 2D_0)\setminus B(x_0, 2d_0)$ respectively.

Setting $G=\varphi^2 u$, then by our choice of $\varphi$ and (4.5), we see that
\begin{eqnarray*} \varphi^2 \Delta_f G &=&\varphi^4 \Delta_f u +\varphi^2 u \Delta_f \varphi^2 +2 \varphi^2(\nabla u\cdot\nabla \varphi^2 )\\
&\ge &\varphi^4 \left( a u^2 R^{a-1} -C u^{3/2} R^{a/2}\right) + 2\varphi^2 u (\Delta_f \varphi^2)-8|\nabla \varphi|^2 G +2 \nabla G\cdot\nabla \varphi^2 \\
&\ge & a G^2R^{a-1} -C G^{3/2} R^{a/2} -CG+2\nabla G\cdot\nabla \varphi^2.
\end{eqnarray*}

Assume $G$ achieves its maximum at some point $p\in B(x_0, 2D_0)$.
If $p\in B(x_0, 2D_0)\setminus B(x_0, 2d_0)$, then it follows from the maximum principle that
$$ 0\ge a G^2(p)R^{a-1}(p)  -C G^{3/2}(p) R^{a/2}(p) -CG(p).$$
On the other hand, noticing that the fact $0<a<1$ and $R$ uniformly bounded from above, implies.
$$ G(p)\le C$$ for some constant $C$ depending on $a$ but independent of $D_0$.

Thus,
$$ \max_{B(x_0, D_0)} u \le \max_{B(x_0, 2D_0)} G\le \max\left\{C,\max\limits_{B(2d_0)}u\right\}\le C'$$
for some $C'>0$ indepedent of $D_0$. Therefore $|Ric|^2\le C R^a$ on $M^4$.

It remains to show $|Rm|\le C$ on $M^4$. However, once we know $\sup_{x\in M} Ric\le C$,
$|Rm|\le C$ follows essentially from the same argument as in the proof of Theorem 3.1. We leave
the detail for the interested reader.

\end{proof}

\begin{lemma}
Let $(M^4, g_{ij}, f)$, which is not Ricci-flat,  be a complete noncompact gradient steady Ricci soliton with $\lim_{x\to \infty} R(x)=0$. Then for each $0<a<1$ and $\mu>0$, there exist
constants $\lambda>0$ and $D>0$ so that function $$v=\frac{|Rm|^2+\lambda|Ric|^2}{R^{a}}$$ satisfies the differential inequality
$$\Delta_f v\geq \mu v-D.$$

\end{lemma}

\begin{proof} By Lemma 2.4 and Theorem 4.1,
\begin{eqnarray*}
\Delta_f v&=&\frac{\Delta_f(|Rm|^2+\lambda|Ric|^2)}{R^a}+vR^a\Delta_f(\frac{1}{R^a})+2\nabla(vR^a)\cdot\nabla(R^{-a})\\
&\geq&\frac{2|\nabla Rm|^2+2\lambda|\nabla Ric|^2}{R^a}-c \frac{|Rm|^2+\lambda|Ric|^2}{R^a}\\
& &+(|Rm|^2+\lambda|Ric|^2)\left[-a\frac{\triangle_f R}{R^{a+1}}+a(a+1)\frac{|\nabla R|^2}{R^{a+2}}\right]\\
& &-4a\frac{|Rm||\nabla Rm||\nabla R|}{R^{a+1}}-4a\lambda\frac{|Ric||\nabla Ric||\nabla R|}{R^{a+1}}.\\
\end{eqnarray*}
By using Cauchy's inequality to terms with $|\nabla R|$,

\begin{eqnarray*}
\Delta_f v&\geq&\frac{2|\nabla Rm|^2+2\lambda|\nabla Ric|^2}{R^a}-c \frac{|Rm|^2+\lambda|Ric|^2}{R^a}\\
& &-\frac{4a}{a+1}\frac{|\nabla Rm|^2}{R^a}-\frac{4a\lambda}{a+1}\frac{|\nabla Ric|^2}{R^a}\\
&\geq& \frac{2\lambda(1-a)}{1+a}\frac{|\nabla Ric|^2}{R^a}-c \frac{|Rm|^2+\lambda|Ric|^2}{R^a}.
\end{eqnarray*}

Now by Proposition 3.1, for some constant $\epsilon>0$, we have
\begin{eqnarray*}
2\epsilon|Rm|^2&\leq&\left(|\nabla Ric|+|Ric|^2+|Rc|\right)^2\\
&\leq& 2 |\nabla Ric|^2+2(|Ric|^2+|Ric|)^2.\\
\end{eqnarray*}

Thus,
\begin{eqnarray*}
\Delta_f v&\geq& \left[\frac{2\epsilon\lambda(1-a)}{1+a}-c\right]\frac{|Rm|^2}{R^a}-\left[2\lambda\frac{1-a}{1+a}(|Ric|+1)^2+c\lambda\right]\frac{|Ric|^2}{R^a}\\
&\ge&\left[\epsilon\lambda(1-a)-c\right](v-\lambda\frac{|Ric|^2}{R^a})-\lambda \left[2(1-a)(|Ric|+1)^2+c\right]\frac{|Ric|^2}{R^a}.
\end{eqnarray*}
Therefore, by Theorem 4.1, for each $0<a<1$ and $\mu>0$ one can choose $\lambda\ge C/(1-a)$, with $C>0$ depending on $\mu$ and sufficiently large, so that
\begin{equation*}
\Delta_f v \ge \mu v-D
\end{equation*}
for some constant $D>0$ depending on $\lambda$.
\end{proof}


\begin{theorem}
Let $(M^4,g_{ij},f)$, which is not Ricci-flat,  be a complete noncompact gradient steady Ricci soliton with $\lim_{r \to \ \infty} R=0$. Suppose R has at most polynomial decay, i.e., $R(x)\ge C/r^k(x)$ outside $B(r_0)$
for some fixed $r_0>1$, some constant $c>0$ and positive integer $k$. Then, for each $0<a<1$, there exists a constant $C$ such that
$$|Rm|\le C R^{a/2}.$$
\end{theorem}
\begin{proof}

Let $p=\frac{k}{2}$. Consider the following function on $\mathbb{R}^{+}$:
$$ {\varphi (t) =\left\{
       \begin{array}{ll}
  \left(\frac{d-t}{d}\right)^{p} \ \ \quad  0\le t\le d\\[5mm]
   \quad 0 \ \ \quad \quad \quad \quad t\ge d.
       \end{array}
    \right.}$$

Next, let $\varphi=\varphi(r(x))$ on $M^4$. Then we have
\begin{eqnarray*}
|\nabla \varphi| &=&\frac{p}{d}\left(\frac{d-r}{d}\right)^{p-1}|\nabla r|=\frac{p}{d-r}\varphi,\\
\triangle_f \varphi&=&-\frac{p}{d}\left(\frac{d-r}{d}\right)^{p-1}\triangle_f r+\frac{p(p-1)}{d^2}\left(\frac{d-r}{d}\right)^{p-2} |\nabla r|^2\\
&=&  \left[-\frac{p}{d-r} \triangle_f r+\frac{p(p-1)}{(d-r)^2}\right]\varphi
\end{eqnarray*}

Consider $w=v-\frac{D}{\mu}$ with $v=\frac{|Rm|^2+\lambda|Ric|^2}{R^{a}}$, $\mu$ and $D$ as in Lemma 4.2. Then,   $w$ satisfies
\begin{eqnarray*}
\triangle_f w\geq\mu w.
\end{eqnarray*}

Let $G$=$\varphi^{2}w$, then outside $B(r_0)$, we have

\begin{align*}\label{eq:abc}
\triangle_f G&=(\Delta_f\varphi^2)w+\varphi^2\Delta_f w+2(\nabla\varphi^2)\cdot\nabla w\\
&\ge\left(2\varphi\Delta_f\varphi+2|\nabla\varphi|^2\right) w+\mu \varphi^2 w+4\varphi\nabla\varphi \cdot \nabla\frac{G}{\varphi^2}\\
&\geq\left(\mu+\frac{2\triangle_f \varphi}{\varphi}-6\frac{|\nabla \varphi|^2}{\varphi^2}\right)G+\frac{4}{\varphi}\langle\nabla G, \nabla \varphi\rangle.\tag{4.6}
\end{align*}

Recall that $G=0$ outside $B(d)$. Now consider a maximum point $q$ of $G$.

\medskip
{\bf Case 1}. $G(q)\le 0$. Then, $\max\limits_{B(d)} w\leq 0.$

\medskip
{\bf Case 2}. $G(q)>0$ and $q\in B(r_0)$. Then, on $\Omega= B((1-\frac{1}{2^{1/p}})d)$, we have
\begin{eqnarray*}
\max\limits_{\Omega} w &\leq&
\max\limits_{\Omega}\frac{1}{\varphi^{2}}\cdot G(q)\\
&\leq& 4 G(q) \\
&\leq& 4 \max\limits_{B(r_0)} w.
\end{eqnarray*}

\medskip
{\bf Case 3}. $G(q)>0$ and $q\notin B(r_0).$ Then, by (4.6) and Lemma 4.1, at $q$ we have
\begin{eqnarray*}
0&\geq&\mu+2\frac{\triangle_f \varphi}{\varphi}-6\frac{|\nabla \varphi|^2}{\varphi^2}\\
&\geq&\mu-2p K_0 \frac{1}{d-r}-(4p^2+2p)\frac{1}{(d-r)^2}
\end{eqnarray*}
for some constant $K_0>0$ depending on $r_0$ and $\max_{B(r_0)} |Ric|$.
Hence $\frac{1}{d-r(q)}>C$ for some constant $C$ depending on $\mu$, $p=k/2$ and $K_0$. Thus,
we have
$$d-r(q)\leq c \eqno(4.7)$$ for some constant $c>0$ independent of $d$.

Therefore,
\begin{eqnarray*}
\max\limits_{\Omega} w &\leq&
\max\limits_{\Omega}\frac{1}{\varphi^{2}}\cdot G(q)\\
&\leq& 4 G(q) \\
&\leq&4 \frac {(d-r(q))^{2p}}{d^{2p}} \frac{\left(|Rm|^2+\lambda|Ric|^2\right)} {R^a}(q)\\
&\leq&C\frac{r^{ak}(q)}{d^{2p}} \\
&\leq&C d^{(a-1)k} \le C
\end{eqnarray*}
for some constant $C>0$ independent of $d$.
Since $d>r_0$ is arbitrary, we obtain $\sup_{M} w \le C$, and hence $|Rm|^2\le C R^{a}$ on $M^4$ for each $0<a<1$.

\end{proof}

\end{document}